\documentclass[12pt]{amsart}
\usepackage{amsfonts,amssymb,amsmath,bm,a4,color}

\numberwithin{equation}{section}
\def\hs{{\mathcal H}^s}
\newcommand{\alq}{\tfrac{{\vv a} + \bm\lambda}{q}}
\newcommand{\Mat}{\operatorname{Mat}}

\newcommand{\GL}{\mathrm{GL}}

\newcommand{\cU}{{\vv U}}

\newtheorem{theorem}{Theorem}[section]
\newtheorem{theoremKM}{Theorem KM \!\!\!\!}

\newtheorem{corollary}{Corollary}[section]
\newtheorem{lemma}{Lemma}[section]

\newtheorem{proposition}{Proposition}[section]
\theoremstyle{remark}
\newtheorem{remark}{Remark}[section]

\def\R{\mathbb{R}}

\def\Q{\mathbb{Q}}
\def\Z{\mathbb{Z}}
\def\N{\mathbb{N}}

\newcommand{\ka}{\kappa}

\def\cH{\mathcal{H}}

\def\cM{\mathcal{M}}

\def\cG{\mathcal{G}}

\def\cU{\mathcal{U}}
\def\cF{\mathcal{F}}

\def\cS{{\mathcal{S}}}

\def\cR{\mathcal{R}}
\def\cG{\mathcal{G}}

\newcommand{\we}{\wedge}

\newcommand{\ve}{\varepsilon}

\newcommand{\vv}[1]{{\mathbf{#1}}}
\newcommand{\mv}[1]{{\bm#1}}

\newcommand{\rank}{\operatorname{rank}}

\newcommand{\dist}{\operatorname{dist}}

\newcommand{\diag}{\operatorname{diag}}

\renewcommand{\tilde}{\widetilde}

\newcommand{\vol}{\operatorname{Vol}}

\newcommand{\Jarnik}{Jarn\' \i k}

\newcommand{\hidden}[1]{}

\begin{document}

\title[Diophantine approximation on manifolds]{Diophantine approximation on curves and the distribution of rational points: divergence theory}

\author{V. Beresnevich \and R.C. Vaughan \and S. Velani \and E. Zorin}

\thanks{VB and SV:  Research supported by EPSRC Programme Grant EP/J018260/1.}
\thanks{EZ:  Research supported by EPSRC Grant EP/M021858/1.}
\thanks{VB, SV and EZ:  University of York, Heslington, York, YO10 5DD, UK}
\thanks{RCV: Pennsylvania State University, University Park, PA 16802-6401, USA}

\date{{\small\today}}

\begin{abstract}
In this paper we develop a new explicit method to studying rational points near manifolds and obtain optimal lower bounds on the number of rational points of bounded height lying at a given distance from an arbitrary non-degenerate curve. This generalises previous results for analytic non-degenerate curves. Furthermore, the main results are also proved in the inhomogeneous setting. Applications of the main theorem include the Khintchine-Jarn\'ik type theorem for divergence for arbitrary non-degenerate curve in $\R^n$.
\end{abstract}

\maketitle

{\footnotesize

}

\vspace*{0ex}

\section{Introduction and statement of results \label{convintro}}

The main goal of this paper is to obtain sharp lower bounds for the number of rational points lying close to an  arbitrary non-degenerate curve in $\R^n$. Motivated by applications to Diophantine approximation on manifolds, such bounds were obtained for planar curves \cite{Beresnevich-Dickinson-Velani-07:MR2373145, Beresnevich-Vaughan-Velani-11:MR2777039, BZ2010} and analytic non-degenerate submanifolds of $\R^n$ \cite{Beresnevich-SDA1}. Out motivation is precisely the same and we shall describe such applications in \S\ref{applicitons} below.

Recall that a real connected analytic submanifold of $\R^n$ is non-degenerate if and only if it is not contained in any hyperplane of $\R^n$ \cite[p.~341]{Kleinbock-Margulis-98:MR1652916}.
One of the main results of \cite{Beresnevich-SDA1}, implies that for any analytic non-degenerate submanifold $\cM\subset\R^n$ of dimension $d$ and codimension $m=n-d$ we have that
\begin{equation}\label{e602}
\#\Big\{\vv p/q\in\Q^n:\ 1\le q\le Q,\ \dist(\vv
p/q,\cM)\le\frac\psi Q\Big\}\ge  C_1 \psi^{m} Q^{d+1}
\end{equation}
for all sufficiently large $Q$ and all real $\psi$ satisfying
\begin{equation}\label{e600}
    C_2Q^{-1/m }< \psi < 1
\end{equation}
where the symbol $\#$ stands for `cardinality' and $C_1$ and $C_2$ are positive constants depending only on the manifold $\cM$ and the dimension $n$ of the space. Furthermore, as shown in \cite[Theorem~7.1]{Beresnevich-SDA1} for analytic non-degenerate curves \eqref{e600} can be relaxed to
\begin{equation}\label{e601}
    C_2 Q^{-\frac{3}{2n-1}}< \psi < 1\,.
\end{equation}
It is believed that the above results for analytic non-degenerate manifolds should hold for arbitrary non-degenerate manifolds. Indeed, this is the case for planar curves, see \cite{Beresnevich-Dickinson-Velani-07:MR2373145, BZ2010}. The main purpose of the present paper is to show this is the case for non-degenerate curves in arbitrary dimensions. Furthermore, we obtain an inhomogeneous extension of \eqref{e602}, which to date is only  known in the case  $n=2$, see \cite{Beresnevich-Vaughan-Velani-11:MR2777039}.

Before we proceed with the statement of results, let us recall the definition of non-degeneracy in the non-analytic case.
Firstly, a map  $\vv f  : \mathcal U  \to \R^n$ defined on an open set $\cU\subset\R^d$ is called \emph{$l$-non--degenerate at} $\vv x\in \mathcal U$ if $\vv f$  is $l$ times continuously differentiable on some sufficiently small ball centred at $\vv x$ and the partial derivatives
of $\vv f$ at $\vv x$ of orders up to $l$ span $\R^n$. The map $\vv f$ is
called \emph{non--degenerate} at $\vv x$ if it is $l$-non--degenerate at $\vv x$ for some $l$; in turn
a manifold $ \cM\subset\R^n$ is said to be non--degenerate at $\vv y\in\cM$ if there is a neighbourhood of $\vv y$ that can be parameterised by a map $\vv f$ non-degenerate at $\vv f^{-1}(\vv y)$. In general, non-degenerate manifolds  are smooth sub-manifolds of $\R^n$ which are sufficiently curved so as to deviate from any hyperplane at a polynomial rate, see \cite[Lemma~1(c)]{B2002}.

\subsection{Results for rational points near manifolds}\label{counting section}

Throughout, $|X|$ is the Lebesgue measure of a measurable subset $X$ of $\R$, $\|\cdot\|_2$ is the Euclidean norm and  $\|\cdot\|_\infty$ is the supremum norm. In what follows, unless otherwise stated, all balls will be considered with respect to the supremum norm.
Let $d,m\in\N$, $n=d+m$ and $\bm f=(f_1,\ldots,f_m)$ be defined and continuously differentiable on a given fixed ball $\mathcal U$ in $\R^d$. The map $\bm f$ naturally gives rise to the $d$-dimensional  manifold
\begin{equation}
\label{monge}
 \cM_{\bm f} := \left\{ (\vv x,\bm f(\vv x )) \in \R^n : \vv x =(x_1, \ldots, x_d) \in \mathcal U  \right\}  \,
\end{equation}
immersed into $\R^n$. By the  Implicit  Function Theorem, any smooth submanifold $\cM$ of $\R^n$ can be (at least locally) defined in  this manner; i.e. with a Monge parametrisation. Hence, in what follows, without loss of generality, we will work with a manifold $\cM$ as in \eqref{monge}.

Given $0<\psi<1$, $Q>1$, a ball $B\subset\cU$ and $\bm\theta=(\bm\lambda,\bm\gamma) \in \R^d \times \R^m$, consider the set
 \begin{equation}
\label{sv1}
\mathcal R(Q,\psi,B,\bm\theta):= \left\{(q,\mathbf a, \mathbf b)  \in \N\times\Z^d \times \Z^m    :
  \begin{array}{l}
  \alq \in  B \, ,~~\tfrac12Q<q\le Q\,,  \\[1ex]
    \|q \bm f\big(\alq\big) - \bm\gamma - {\mathbf b} \|_\infty < \psi
    \end{array}
\right\}\,.
\end{equation}
Also we define
\begin{equation} \label{def_big_Delta}
\Delta(Q,\psi,B,\bm\theta,\rho):=\bigcup_{(q,\vv a,\vv b)\in\cR(Q,\psi,B,\bm\theta)}B(\alq,\rho),
\end{equation}
where $B(\vv x,\rho)$ denotes the ball  in $\R^d$ centred at $\vv x$ and of radius $\rho$.
Clearly, $\cR(Q,\tfrac12\psi,B,\bm\theta)$ contains shifted rational points
\begin{equation}\label{points}
\textstyle{\big( \frac{a_1+\lambda_1}{q}, \dots, \frac{a_d+\lambda_d}{q}, \frac{b_1+\gamma_1}{q}, \ldots, \frac{b_m+\gamma_m}{q} \big) }  \in \R^n
\end{equation}
with denominators  $q$ in $[\tfrac12Q,Q]$ that  lie within the  $2\psi/Q$-neighbourhood  of $\vv f(B)\subset \cM_{\bm f}$, where $\vv f(x):=(x,\mv f(x))$. Thus, an appropriate lower bound on the cardinality  of $\cR(Q,\tfrac12\psi,B,\bm\theta)$ would yield \eqref{e602}. With this in mind, the following statement represents our key result.

\begin{theorem} \label{t3}
Let $\bm\theta\in\R^{n}$, $\bm f=(f_1,\ldots,f_{n-1})$ be a map of one real variable such that $x\mapsto \vv f(x):=(x,\bm f(x))$ is non-degenerate at some point $x_0\in\R$. Then, there exists a sufficiently small interval $\cU$ centred at $x_0$ and constants $C_0,K_0>0$ $($depending on $n$, $\vv f$ and $x_0$ only$)$ such that for any subinterval $B\subset \cU$
there is a constant $Q_B$ depending on $n$, $\bm f$ and $B$ only such that for any integer $Q\ge Q_B$ and any $\psi$ satisfying
\begin{equation} \label{theo_general_psi_db2}
K_0Q^{-\frac{3}{2n-1}}\le \psi< 1
\end{equation}
we have
\begin{equation} \label{theo_general_result_two}
|\Delta(Q,\psi,B,\bm\theta,\rho)|\geq \tfrac12|B|
\end{equation}
where
\begin{equation} \label{def_rho}
\rho=\frac{C_0}{\psi^{n-1} Q^{2}}\,.
\end{equation}
\end{theorem}

The following desired counting result is an immediate consequence of the theorem.

\begin{corollary}\label{cor1}
Assuming $\bm\theta$, $\cU$, $\bm f$, $x_0$, $B$, $C_0$, $\psi$ and $Q$ are the same as in Theorem~\ref{t3}, we have that
\begin{equation}\label{lb}
\#\cR(Q,\psi,B,\bm\theta) \ \ge \  \frac{|B|}{4C_0}~\psi^{n-1}  Q^{2}  \, .
\end{equation}
\end{corollary}

The proof of Corollary~\ref{cor1} is the same as that of Corollary~1.5 in \cite{Beresnevich-SDA1}.

\begin{remark}
The constant $C_0$ appearing in the above statements will be defined within \eqref{Crho} below and can be expressed explicitly in terms of certain parameters associated with $\vv f$ and $x_0$.
\end{remark}

\begin{remark}
Lower and matching upper bounds for rational points near non-degenerate planar curves can be found in \cite{Beresnevich-Dickinson-Velani-07:MR2373145, Beresnevich-Vaughan-Velani-11:MR2777039, BZ2010,  Chow, Huang2015, Gafni}.
In the homogeneous case (i.e. when $\bm\theta = \bm0$), the lower bound  given by  \eqref{lb} is established  in~\cite{Beresnevich-SDA1} for analytic non-degenerate curves embedded in $\R^n$.
The key outcome of this paper is thus the removal of the analytic assumption, which is done upon introducing a new technique for detecting rational points near manifolds. This technique enables us to perform explicit analysis of certain conditions within the so-call Quantitative Non-Divergence estimate  of Kleinbock and Margulis (see Section \ref{QND} below)  that underpins the proof of our main result.
\end{remark}

\subsection{Simultaneous Diophantine approximation on manifolds}\label{applicitons}

Given a function  $\psi: (0,+\infty) \to (0,+\infty) $ and   a point   $\bm\theta =(\theta_1,\ldots,\theta_n) \in \R^n$,  let  $\cS_n(\psi,\bm\theta)$ denote the set of $\vv y=(y_1,\dots,y_n)\in\R^n$ for which  there exists  infinitely many
$(q,p_1,\dots,p_n)\in\N\times\Z^n$ such that
$$
\max_{1\le i\le n}|q y_i- \theta_i-p_i|<\psi(q) \, .
$$
If $\bm\theta =\vv0$ then the corresponding set  $\cS_n(\psi):=\cS_n(\psi,\bm0)$ is the usual homogeneous set of simultaneously $\psi$-approximable points in $\R^n$. In the case $\psi$ is $\psi_{\tau}:r\to r^{-\tau}$ with $\tau>0$, let us write
$\cS_n(\tau,\bm\theta)$  for $\cS_n(\psi,\bm\theta)$ and $\cS_n(\tau)$  for $\cS_n(\tau,\bm0)$. Recall that, by Dirichlet's theorem, $\cS_n(\tau)= \R^n $ for $\tau \leq 1/n$.

As an application of our main result (Theorem~\ref{t3}) we have the following statement concerning the `size' of the set of simultaneously $\psi$-approximable points restricted to lie on a curve in $\R^n$.

\begin{theorem}\label{t1}
Let $\bm\theta \in \R^n$ and  $\psi: (0,+\infty) \to (0,+\infty) $ be any monotonic function such that $q\psi(q)^{(2n-1)/3}\to\infty$ as $q\to\infty$. Let $\cM$ be any non-degenerate curve in $\R^n$. Then for any $s\in\R$ satisfying $\tfrac12 < s \le 1 $  we have that
\begin{equation*}
\hs \big( \cS_n(\psi,\bm\theta)\cap\cM \big) = \cH^s(\cM)   \qquad \text{when} \qquad  \sum_{q=1}^{\infty}  q^n\left( \textstyle{\frac{\psi(q)}{q}}  \right)^{s+n-1}\ = \ \infty \, .
\end{equation*}
In particular, if
$$
\tau(\psi):=\liminf_{q\to\infty}\frac{-\log \psi(q)}{\log q}\,,
$$
the lower order of $1/\psi$ at infinity, satisfies the inequalities $n\le \tau(\psi)<3/(2n-1)$,
then
\begin{equation}\label{e:135-}
\dim \big(\cS_n(\psi,\bm\theta)\cap\cM\big) \ge \frac{n+1}{\tau(\psi)+1}-n+1 .
\end{equation}
\end{theorem}

\bigskip

\begin{remark}
In the case $s<d$ we have that $\cH^s(\cM)=\infty$ and thus Theorem~\ref{t1} represents an analogue of Jarn\'\i k's theorem \cite{Jarnik}. When $s=d$, Theorem~\ref{t1} represent an analogue of Khintchine's theorem \cite{Khin26} for curves.
\end{remark}

\begin{remark}
Theorem~\ref{t1} was previously proved for planar curves, see \cite[Theorem~3]{Beresnevich-Dickinson-Velani-07:MR2373145}, \cite[Theorem 1]{Beresnevich-Vaughan-Velani-11:MR2777039} and \cite[Theorem~4]{BZ2010}.  For $n>2$, Theorem~\ref{t1} was previously established in the homogeneous case for non-degenerate curves that are additionally assumed to be analytic \cite[Theorem~7.2]{Beresnevich-SDA1}.
 Most recently, it was proved in \cite{Beresnevich-Lee-Vaughan-Velani-17}, that if the stronger inequality  $n\le \tau(\psi)<1/(n-1)$ holds and the upper order of $1/\psi$ equals the lower of order of $1/\psi$;  namely that $$
\limsup_{q\to\infty}\frac{-\log \psi(q)}{\log q}=\liminf_{q\to\infty}\frac{-\log \psi(q)}{\log q}\,,
$$
then the  lower bound dimension statement  \eqref{e:135-} is valid in the homogeneous case for arbitrary $C^2$ curves (including degenerate ones) in $\R^n$. To date, the complementary convergence theory for curves is only known in full when $n=2$ -- see  \cite{Vaughan-Velani-06:MR2242634}.  For submanifolds of $\R^n$ of dimension $\ge2$,  see \cite{Beresnevich-Vaughan-Velani-Zorin}, \cite{Huang2}, \cite{Simmons} and references within for various convergence results. For a general background to previous results and what one expects to be able to prove, see \cite[\S1.6]{VicFel}.
\end{remark}

\begin{remark}
Theorem~\ref{t1} can be extended to submanifolds of $\R^n$ of any dimension by making use of a slicing technique due essentially to  Pyartli~\cite{Pyartly-1969}. Indeed, by using this technique one can see that if a submanifold $\cM$ of $\R^n$ of any dimension admits a fibering into non-degenerate curves, then Theorem~\ref{t1} extends to such a manifold. Naturally, one example of a class of manifolds admitting a fibering into non-degenerate curves is analytic manifolds (see the Fibering Lemma in~\cite{Beresnevich-SDA2}). There are of course classes of non-analytic manifolds that admit fibering into non-degenerate curves -- see \cite{Pyartly-1969} for concrete examples.
\end{remark}

In short, Theorem~\ref{t3} establishes a ubiquitous system of shifted rational points \eqref{points} near $\cM_{\mv f}$. Ubiquity \cite{Beresnevich-Dickinson-Velani-06:MR2184760} is a well developed mechanism for proving divergence statements such as Theorem~\ref{t1} above. In particular, the deduction of Theorem~\ref{t1} from Theorem~\ref{t3} follows  the blue print  presented in \cite[pp.196--199]{Beresnevich-SDA1} that generalises the `planar' arguments in \cite[\S7]{Beresnevich-Dickinson-Velani-07:MR2373145} to higher dimensions. The necessary modifications necessary (for proving Theorem~\ref{t1} from Theorem~\ref{t3}) are obvious and essentially account for the shift in the numerators of the rational points to reflect the inhomogeneous nature of the problem under consideration.  The details are left to the reader and thus the rest of this paper is devoted to the proof of Theorem~\ref{t3}.

\section{Detecting rational points near manifolds}
\label{detect}

In this section we introduce an alternative to the method of \cite{Beresnevich-SDA1} for detecting rational points near manifolds. As before, we assume that $\cM$ is given by its Monge parameterisation \eqref{monge}.
In this section we will be making no assumptions about the dimension $d$ of $\cM$.
Without loss of generality we will assume that there exists a constant $M>0$ such that
\begin{equation} \label{lemma_translation_M}
\max_{1\le k\le m}\,\max_{1\le i,j\le d}\,\,\sup_{\vv x\in \cU}\,\left|\frac{\partial^2 f_k(\vv x)}{\partial x_i\partial x_j}\right|\leq M\,.
\end{equation}
Define the following $m$ auxiliary functions of $\vv x=(x_1,\dots,x_d)$\,:
\begin{equation}\label{g_j}
g_j:=f_j-\sum_{i=1}^dx_i\frac{\partial f_j}{\partial x_i}\qquad\quad(1\le j\le m)
\end{equation}
and the following $(n+1)\times(n+1)$ matrix

\begin{equation} \label{section_DP_def_gG}
G=G({\vv x}):=\left(
\begin{array}{cccccccc}
g_1&\dfrac{\partial f_1}{\partial x_1} & \dots & \dfrac{\partial f_1}{\partial x_d}& -1&0&\dots& 0\\[2ex]
\vdots&\vdots&\ddots&\vdots&\vdots&&\ddots&\vdots\\[1ex]
g_m&\dfrac{\partial f_m}{\partial x_1} & \dots & \dfrac{\partial f_m}{\partial x_d}& 0&0&\dots& -1\\[2ex]
x_1&-1 & \dots & 0& 0&0&\dots& 0\\[1ex]
\vdots&\vdots & \ddots &\vdots & \vdots&\vdots&& \vdots\\[1ex]
x_d&0 & \dots &-1 & 0&0&\dots& 0\\[1ex]
1&0 & \dots & 0 & 0&0&\dots& 0
\end{array}
\right)\,.
\end{equation}

\noindent Next, given positive $c,Q,\psi$, let
\begin{equation} \label{section_DP_def_g}
g=g(c,Q,\psi):=\diag\Big\{\underbrace{\psi,\dots,\psi}_m,\underbrace{(\psi^m Q)^{-1/d},\dots,
(\psi^m Q)^{-1/d}}_d,c Q\Big\}
\end{equation}
be a diagonal matrix. Finally, define the set
\begin{equation} \label{def_Bdelta}
\cG(c,Q,\psi):=\Big\{\vv x\in \cU:\delta\big(g^{-1}G({\vv x})\Z^{n+1}\big)\ge 1\Big\}\,,
\end{equation}
where for a given lattice $\Lambda\subset\R^{n+1}$
\begin{equation}\label{vb0}
\delta\big(\Lambda\big):=\inf_{\vv v\in\Lambda\setminus\{\vv0\}}\|\vv v\|_\infty\,.
\end{equation}
Given a set $S\subset\R^d$ and a real number $\rho>0$, $S^\rho$ will denote its `$\rho$-interior'; that is the set of  $\vv x\in S$ such that $B(\vv x,\rho)\subset S$.

\begin{lemma} \label{lemma_translation}
Let $Q,\psi>0$ be given and satisfy the following inequality
\begin{equation} \label{section_DP_psi_is_big}
\psi\ge Q^{-\frac{d+2}{2m+d}}\,.
\end{equation}
Let $\cU$ be a ball in $\R^d$ and let $\mv f=(f_1,\dots,f_m):\cU\to\R^m$ be a $C^2$ map such that \eqref{lemma_translation_M} is satisfied for some $M>0$. Let $c>0$, $\bm\theta:=(\bm\lambda,\bm\gamma) \in \R^d \times \R^m$ and
\begin{equation}\label{rho}
\rho:=\frac{1}{2c}\,(\psi^m Q^{d+1})^{-1/d}\,.
\end{equation}
Then, for any ${\vv x}=(x_1,\dots,x_d)\in\cG(c,Q,\psi)\cap\cU^\rho$
there exists an integer point $(q,a_1,\dots,a_d,b_1,\dots,b_m)\in\Z^{n+1}$ such that
\begin{align}
&2(n+1)Q\,\,<\,\,q \,\, <\,\, 4(n+1)Q, \label{lemma_translation_conclusion_q}\\[2ex]
&|qx_i-a_i-\lambda_i|\,\,<\,\,\frac{n+1}{c}\,(\psi^m Q)^{-1/d}\label{lemma_translation_conclusion_x}\hspace*{8.2ex}(1\le i\le d),
\end{align}
and
\begin{align}
&\label{lemma_translation_conclusion_f}
\left|qf_j\left(\frac{a_1+\lambda_1}q,\dots, \frac{a_d+\lambda_d}q\right)-b_{j}-\gamma_{j}\right|\\[1ex]
&\nonumber\hspace*{13.5ex}\,\,<\,\,\left(1+\frac{Md^2}{2c}\right)\frac{n+1}{c}\,\psi\qquad (1\le j\le m).
\end{align}
\end{lemma}

\bigskip

\begin{remark}  \label{remSV}
The fact that $\vv x$ is restricted to lie in $\cU^\rho$ means that  $B(\vv x,\rho)\subset\cU$ and this ensures that the shifted rational point $\left(\frac{a_1+\lambda_1}q,\dots, \frac{a_d+\lambda_d}q\right)$ lies in $\cU$. Indeed, once \eqref{lemma_translation_conclusion_q} and \eqref{lemma_translation_conclusion_x} are met, the associated shifted rational  point lies in  $B(\vv x,\rho)$ and hence in $\cU$. It is not difficult to see from \eqref{section_DP_psi_is_big} that $\rho\to0$ as $Q\to\infty$ uniformly in $\psi$ and thus considering points $\vv x$ lying in $ \cU^\rho$  rather than $\cU$ is not particularly restrictive.
\end{remark}

\begin{proof}
Fix any $\vv x\in \cG(c,Q,\psi)\cap\cU^\rho$ and consider the lattice
$$
\Lambda:=g^{-1}G({\vv x})\Z^{n+1}\,.
$$
Let $\mu_1,\dots,\mu_{n+1}$ be the successive Minkowski minima of $\Lambda$ with respect to the body
$$
B:=[-1,1]^{n+1}.
$$
By definition, $\mu_i$ is the infimum of all $x>0$ such that $\rank(\Lambda\cap xB)\ge i$, where $xB:=[-x,x]^{n+1}$. In particular, we have that $\mu_1\le\ldots\le\mu_{n+1}$.
By Minkowski's theorem on successive minima, we have that
$$
\frac{2^{n+1}}{(n+1)!}\le \frac{\vol(B)}{{\rm covol}(\Lambda)}\prod_{i=1}^{n+1}\mu_i  \le 2^{n+1}\,.
$$
Observe, on using \eqref{section_DP_def_gG} and \eqref{section_DP_def_g},  that the covolume of $\Lambda$ is $c^{-1}$ and that the volume of $B$ is $2^{n+1}$. Hence
$$
c\,\prod_{i=1}^{n+1}\mu_i  \le 1\,.
$$
Further, by the assumption that $\vv x\in \cG(c,Q,\psi)$, we have that $\mu_1\ge1$. This follows form \eqref{def_Bdelta}. Hence,
$$\mu_{n+1}\le c^{-1}\prod_{i=1}^{n}\mu_i^{-1}\le c^{-1} \, .  $$ Therefore there exists a basis of $\Lambda$, say $\vv v_1,\dots,\vv v_{n+1}$, lying in $c^{-1}B$, that is
\begin{equation}\label{e1}
\|\vv v_i\|_\infty\le c^{-1}\qquad(1\le i\le n+1).
\end{equation}
Let
$$
\bm\omega:=(\omega_0,\omega_1,\dots,\omega_n)   \in \R^{n+1} \,,
$$
where
$$
\omega_0 :=3(n+1)Q\,,
$$
\begin{equation}\label{g2}
\omega_i:=\lambda_i+\omega_0x_i\qquad(1\le i\le d)
\end{equation}
and
\begin{equation}\label{g3}
\omega_{d+j}:=\gamma_j+\omega_0f_j(\vv x)\qquad(1\le j\le m)\,.
\end{equation}
Since $\vv v_1,\dots,\vv v_{n+1}$ are linearly independent, there exist unique real parameters $\eta_1,\dots,\eta_{n+1}$ such that
\begin{equation}\label{e2}
-g^{-1}  \, G({\vv x}) \, \bm\omega \, = \, \sum_{i=1}^{n+1}\eta_i\vv v_i\,.
\end{equation}
Let $t_1,\dots,t_{n+1}$ be any collection of integers, not all zeros, such that
\begin{equation}\label{g1}
|\eta_i-t_i|\le1 \qquad (1\le i\le n+1).
\end{equation}
The existence of such integers is obvious.
Define
$$
\vv v:=\sum_{i=1}^{n+1}t_i\vv v_i\,.
$$
Since the  $t_i$'s are integers and not all of them are zero, we have that $\vv v\in\Lambda\setminus\{\vv0\}$. Hence, by the definition of $\Lambda$, there exists a non-zero integer point $\vv p\in\Z^{n+1}$, which we will write as $(q,a_1,\dots,a_d,b_1,\dots,b_m)^t$, such that
$$
\vv v=g^{-1}G({\vv x})  \, \vv p\,.
$$
Then, using \eqref{e1}, \eqref{e2} and \eqref{g1}, we find  that

\begin{align}
\label{vb10}\left\|g^{-1}G({\vv x})(\vv p+\bm\omega)\right\|_\infty &= \
\left\|g^{-1}G({\vv x})\vv p+g^{-1}G({\vv x})\bm\omega\right\|_\infty\\[2ex]
\nonumber &= \ \left\|\vv v+g^{-1}G({\vv x})\bm\omega\right\|_\infty\\[2ex]
\nonumber &= \ \left\|\sum_{i=1}^{n+1}t_i\vv v_i-\sum_{i=1}^{n+1}\eta_i\vv v_i\right\|_\infty\\[2ex]
\nonumber & \ \le \sum_{i=1}^{n+1}|t_i-\eta_i|\cdot\left\|\vv v_i\right\|_\infty\\[2ex]
\nonumber & \ \le c^{-1}(n+1)\,.
\end{align}

\noindent Observe that the last coordinate of the vector
\begin{equation}\label{vb11}
g^{-1}G({\vv x})(\vv p+\bm\omega)
\end{equation}

\noindent is $(cQ)^{-1}(q+\omega_0)$, which by \eqref{vb10} is $\le c^{-1}(n+1)$ in absolute value. Hence
$|q+\omega_0|\le (n+1)Q$ and since $\omega_0:=3(n+1)Q$, inequalities \eqref{lemma_translation_conclusion_q} readily follow.

Furthermore, for $i\in\{1,\dots,d\}$ the $m+i$ coordinate of \eqref{vb11} is
\begin{align*}
(\psi^m Q)^{1/d}\Big((q+\omega_0)x_i-(a_i+\omega_i)\Big) & \stackrel{\eqref{g2}}{=}
(\psi^m Q)^{1/d}\Big((q+\omega_0)x_i-(a_i+\lambda_i+\omega_0x_i)\Big)\\[1ex]
&  \ \, = \ \, (\psi^m Q)^{1/d}(qx_i-a_i-\lambda_i)\,.
\end{align*}
By \eqref{vb10} again, we have that $|(\psi^m Q)^{1/d}(qx_i-a_i-\lambda_i)|\le c^{-1}(n+1)$, whence inequalities \eqref{lemma_translation_conclusion_x} follow.

It now  remains to verify \eqref{lemma_translation_conclusion_f}.  With this in mind,  for $j\in\{1,\dots,m\}$ the $j$-th coordinate of \eqref{vb11} equals
$$
\psi^{-1}\left((q+\omega_0)g_j(\vv x)+\sum_{i=1}^d(a_i+\omega_i)\frac{\partial f_j(\vv x)}{\partial x_i}-(b_j+\omega_{d+j})\right)
$$
and by \eqref{g2} and \eqref{g3} this is equivalent to
$$
\psi^{-1}\left((q+\omega_0)g_j(\vv x)+\sum_{i=1}^d(a_i+\lambda_i+\omega_0x_i)\frac{\partial f_j(\vv x)}{\partial x_i}-\Big(b_j+\gamma_{d+j}+\omega_0f_j(\vv x)\Big)\right)\,.
$$
Now on using the expression for $g_j(\vv x)$ from \eqref{g_j},  we can  simplify the above to
$$
\psi^{-1}\left(qg_j(\vv x)+\sum_{i=1}^d(a_i+\lambda_i)\frac{\partial f_j(\vv x)}{\partial x_i}-b_j-\gamma_j\right)\,.
$$
Once again, by \eqref{vb10} this is $\le c^{-1}(n+1)$ in absolute value and  so it follows that
$$
\left|qg_j(\vv x)+\sum_{i=1}^d(a_i+\lambda_i)\frac{\partial f_j(\vv x)}{\partial x_i}-b_j-\gamma_j\right|<c^{-1}(n+1)\psi\,.
$$
Using the expression for $g_j(\vv x)$ given by \eqref{g_j}, we obtain that
\begin{align}
\label{lemma_translation_intermediate_one}
\left|qf_j(\vv x)+\sum_{i=1}^d(a_i+\lambda_i-qx_i)\frac{\partial f_j(\vv x)}{\partial x_i}-
b_j-\gamma_j\right|<c^{-1}(n+1)\psi\,.
\end{align}

We are now ready to establish \eqref{lemma_translation_conclusion_f}.  As already mentioned in  Remark  \ref{remSV}, it follows via \eqref{lemma_translation_conclusion_q} and \eqref{lemma_translation_conclusion_x} that for any point  $\vv x\in \cU^\rho$
$$
\left(\tfrac{a_1+\lambda_1}{q},\dots,\tfrac{a_d+\lambda_d}{q}\right)\in\cU.
$$
Hence, on using Taylor's expansion to the second order followed by the triangle inequality, for any $j\in\{1,\dots,m\}$ we obtain that
\begin{align*}
\left|qf_j\left(\frac{a_1+\lambda_1}{q},\dots,\frac{a_d+\lambda_d}{q}\right)-b_j-\gamma_j\right|&\\[3ex]
 & \hspace*{-30ex} = \ \left|q\left(f_j(\vv x)+\sum_{i=1}^d\frac{\partial f_j(\vv x)}{\partial x_i}\left(\frac{a_i+\lambda_i}{q}-x_i\right)\right.\right. \\[2ex]
& \hspace*{-20ex} + \ \left.\left.\sum_{i,l=1}^d \frac{\partial^2f_j(\widetilde{\vv x})}{\partial x_{i}\partial x_{l}} \left(\frac{a_{i}+\lambda_{i}}{q}-x_{i}\right)\left(\frac{a_{l}+\lambda_l}{q} -x_{l}\right) \right)-b_j-\gamma_j\right| &\\[3ex]
 & \hspace*{-30ex}  \le \   \left|qf_j(\vv x)+\sum_{i=1}^d(a_i+\lambda_i-qx_i)\frac{\partial f_j(\vv x)}{\partial x_i}-b_j-\gamma_j\right|\\[2ex]
&\hspace*{-20ex} +  \  \  \left|\frac1q\sum_{i,l=1}^d \frac{\partial^2f_j(\widetilde{\vv x})}{\partial x_{i}\partial x_{l}} \left(a_{i}+\lambda_{i}-qx_{i}\right)\left(a_{l}+\lambda_l -qx_{l}\right) \right|
\end{align*}

\noindent This together with
\eqref{lemma_translation_M}, \eqref{lemma_translation_conclusion_x} and \eqref{lemma_translation_intermediate_one}, implies that

\begin{align*}
\left|qf_j\left(\frac{a_1+\lambda_1}{q},\dots,\frac{a_d+\lambda_d}{q}\right)-b_j-\gamma_j\right|&\\[3ex]
&  \hspace*{-14ex}  \le \ c^{-1}(n+1)\psi+\frac1qMd^2\left(c^{-1}(n+1)(\psi^m Q)^{-1/d}\right)^2\\[2ex]
   &  \hspace*{-15ex}   \   \stackrel{\eqref{lemma_translation_conclusion_q}}{\le} c^{-1}(n+1)\psi+\frac{Md^2\left(c^{-1}(n+1)(\psi^m Q)^{-1/d}\right)^2}{2(n+1)Q}\\[2ex]
&  \hspace*{-15ex}  \   \stackrel{\eqref{section_DP_psi_is_big}}{\le} \left(1+\frac{Md^2}{2c}\right)c^{-1}(n+1)\psi.
\end{align*}

\noindent This verifies \eqref{lemma_translation_conclusion_f} and thereby  completes the proof of the lemma.
\end{proof}

\medskip

We will make direct use of the following variant of  Lemma~\ref{lemma_translation}.

\medskip

\begin{corollary} \label{lemma_translation_two}
Let $c,M,\tilde Q,\tilde \psi>0$ be given such that
\begin{equation} \label{vb678}
\tilde \psi\ge K_0 \,\tilde Q^{-\frac{d+2}{2m+d}}\, \quad\text{with } \quad K_0\ge(4(n+1))^{\frac{d+2}{2m+d}}\left(1+\frac{Md^2}{2c}\right)\frac{n+1}{c}\,.
\end{equation}
Let $\cU$ be a ball in $\R^d$ and let $\mv f=(f_1,\dots,f_m):\cU\to\R^m$ be a $C^2$ map such that \eqref{lemma_translation_M} is satisfied. Let $\bm\theta=(\bm\lambda,\bm\gamma) \in \R^d \times \R^m$ and let
\begin{align}
\label{Q*}
& Q:=\frac{\tilde Q}{4(n+1)} \   ,   \qquad  \quad\psi:=\frac{\tilde \psi}{\left(1+\frac{Md^2}{2c}\right)\frac{n+1}{c}} \  \ ,\quad   \\[2ex]
&   \rho:=\frac{1}{2c}\,(\psi^m Q^{d+1})^{-1/d}=
C_0(\tilde \psi^m \tilde Q^{d+1})^{-1/d}  \nonumber
\end{align}
where
\begin{equation}\label{Crho}
C_0:=\frac{1}{2c}\,\left(\big(4(n+1)\big)^{d+1}\left(\Big(1+\frac{Md^2}{2c}\Big)\frac{n+1}{c}\right)^m\right)^{1/d}\,.
\end{equation}
Then for any ball $B\subset\cU$,  we have that
\begin{equation}\label{vb234}
\cG(c,Q,\psi)\cap B^\rho\subset\Delta(\tilde Q,\tilde \psi,B,\bm\theta,\rho)\,,
\end{equation}
where $\Delta(\tilde Q,\tilde \psi,B,\bm\theta,\rho)$ is defined as in \eqref{def_big_Delta} .
\end{corollary}

\medskip

\begin{proof}
First, observe that \eqref{vb678} implies \eqref{section_DP_psi_is_big}. Then if $\vv x\in \cG(c,Q,\psi)\cap B^\rho$, it follows by Lemma~\ref{lemma_translation} that  there exists an integer point
  $(q,a_1,\dots,a_d,b_1,\dots,b_m)\in\Z^{n+1}$ satisfying \eqref{lemma_translation_conclusion_q}--\eqref{lemma_translation_conclusion_f}.
By \eqref{Q*}, condition \eqref{lemma_translation_conclusion_q} translates into $\tfrac12\tilde Q< q < \tilde Q$; condition \eqref{lemma_translation_conclusion_f} translates into $\|q \bm f\big(\alq\big) - \bm\gamma - {\mathbf b} \|_\infty < \tilde \psi$; and condition \eqref{lemma_translation_conclusion_x} together with the fact that $q>\tfrac12\tilde Q$ imply that $\vv x\in B(\alq,\rho)$. Thus, in view of \eqref{sv1} and \eqref{def_big_Delta}, where $Q$ and $\psi$ are replaced with $\tilde Q$ and $\tilde \psi$, to complete the proof of the corollary it remains to show that $\alq\in B$. This is trivially the case since  $\vv x\in B^\rho$ and $\vv x\in B(\alq,\rho)$.
\end{proof}

In the light of the above corollary our strategy for
proving Theorem~\ref{t3} will be to find a suitable constant $c>0$ and a ball $\cU$ centred at a given point $\vv x_0$ such that for any ball $B\subset\cU$ and any sufficiently large $Q \in \N$ the Lebesgue measure of $\cG(c,Q,\psi)\cap B^\rho$ is at least a constant times the Lebesgue measure of $B$. In this paper we establish precisely such a statement  in the case of non-degenerate curves.


\begin{theorem} \label{t4}
Let $\bm f=(f_1,\ldots,f_{n-1})$ be a map of one real variable such that $x\mapsto \vv f(x)=(x,\mv f(x))$ is non-degenerate at some point $x_0\in\R$. Fix any $\kappa>0$. Then, there exists a sufficiently small open interval $\cU$ centred at $x_0$ and constants $c,K_0>0$ such that for any subinterval $B\subset \cU$ there is a constant $Q_B$ depending on $n$, $\vv f$, $\kappa$ and $B$ only such that for any integer $Q\ge Q_B$ and any $\psi\in\R$ satisfying \eqref{theo_general_psi_db2} we have that
\begin{equation} \label{theo_general_result_two+}
|B\setminus\cG(c,Q,\psi)|\leq \ka\,|B|\,.
\end{equation}
\end{theorem}

\bigskip

\begin{proof}[Proof of Theorem~\ref{t3} modulo Theorem~\ref{t4}.]
We shall use Corollary~\ref{lemma_translation_two} with the same $\bm f$ and $\bm\theta$ as in the statement of Theorem~\ref{t3}. The fact that $\vv f$ is non-degenerate at $x_0$ implies  that $\vv f$ is at least twice continuously differentiable on a sufficiently small neighborhood $\cU$ of $x_0$. Hence the existence of the constant $M$ satisfying \eqref{lemma_translation_M} follows on taking $\cU$ sufficiently small so that $\vv f$ is $C^2$ on the closure of $\cU$. Shrink $\cU$ further if necessary and choose $c,K_0>0$ such that the conclusions of Theorem~\ref{t4} hold with $\kappa=\tfrac13$. In  particular, let $B\subset\cU$ be any subinterval, $Q>Q_B$ and $\psi$ satisfy \eqref{theo_general_psi_db2}. Then
we have that
\begin{equation}\label{vb765}
|B\setminus\cG(c,Q,\psi)|\leq \tfrac13\,|B|\,.
\end{equation}
Assuming without loss of generality that $K_0$ is at least as in \eqref{vb678}, we observe that Corollary~\ref{lemma_translation_two} is applicable. In particular, by \eqref{vb234}, we have that
\begin{equation} \label{lb_Delta_1}
|\Delta(\tilde Q,\tilde \psi,B,\bm\theta,\rho)|\geq |\cG(c,Q,\psi)\cap B|-|B\setminus B^{\rho}|,
\end{equation}
where $\tilde Q$, $\tilde \psi$ and $\rho$ are defined by~\eqref{Q*}.
It follows  from the definition of $B^\rho$ that $B\setminus B^{\rho}$ is a union of two intervals, each of length $\le \rho$. Therefore
$$
|B\setminus B^{\rho}|\leq 2\rho.
$$
The lower bound in~\eqref{theo_general_psi_db2} implies that $\psi^{n-1}Q^2>K_0^{n-1}Q^{1/2}$. Hence, for any \begin{equation} \label{def_lb}
Q\ge\left(\frac{12\cdot C_0}{K_0^{n-1}|B|}\right)^2,
\end{equation}
we have that
$
\rho
\leq\frac{1}{12}|B|
$
and consequently
\begin{equation} \label{ub_B_rho}
|B\setminus B^{\rho}|<\frac{1}{6}|B|.
\end{equation}
Thus, by \eqref{vb765}, \eqref{lb_Delta_1} and \eqref{ub_B_rho}, we obtain that

$$
|\Delta(Q,\psi,B,\bm\theta,\rho)|\ge\tfrac12|B|
$$
provided that
$$
Q\ge Q^*_B:=\max\left\{Q_B,~\left(\frac{12\cdot C_0}{K_0^{n-1}|B|}\right)^2\right\}\,.
$$
This verifies \eqref{theo_general_result_two} and completes the proof of Theorem~\ref{t3} modulo Theorem~\ref{t4}.

\end{proof}

\section{Quantitative non-divergence}\label{QND}

In what follows we give a simplified account of the theory developed by Kleinbock and Margulis in \cite{Kleinbock-Margulis-98:MR1652916} by restricting ourselves to functions of one variable.
Let $U$ be an open subset of
$\R$, $f:U\to\R$ be a continuous function and let
$C,\alpha>0$. The function $f$ is called {\it $(C,\alpha)$-good on
$U$}\/ if for any open ball (interval) $B\subset U$ the following is satisfied
\begin{equation}\label{e:089}
    \forall\,\ve > 0\qquad \left|\Big\{x\in B : |f(x)| < \ve \,\sup_{x\in B}|f(x)|\,\Big\}\right|\
    \le\
C\,\ve^\alpha \,|B|.
\end{equation}
Given $\lambda>0$ and a ball $B=B(x_0,r)\subset\R$
centred at $x_0$ of radius $r$, $\lambda B$ will denote the `scaled' ball
$B(x_0,\lambda r)$. Given $\vv v_1,\dots,\vv v_r\in\R^{n+1}$ we shall write
$\|\vv v\we\dots\we\vv v_r\|_\infty$ for the supremum norm of the multivector
$\vv v_1\we\dots\we\vv v_r$.  By definition, this is the maximum of the absolute values of the coordinates of
$\vv v_1\we\dots\we\vv v_r$ in the standard basis.  These coordinates are all the possible $r\times r$
minors of the matrix $\Gamma$ composed of the vectors $\vv v_i$ as its columns, see \cite{Schmidt-1980}.
Also, given an $(n+1)\times r$ matrix $\Gamma$, we will write $\|\Gamma\|_\infty$ for $\|\vv v_1\we\dots\we\vv v_r\|_\infty$, where $\vv v_1,\dots,\vv v_r$ are the columns of $\Gamma$.

We will use the following slightly simplified version of  \cite[Theorem~5.2]{Kleinbock-Margulis-98:MR1652916} due to  Kleinbock and Margulis.

\begin{theoremKM}[Quantitative Non-Divergence]
Let $n\in\N$, $C,\alpha>0$ and $0<\rho\le 1/(n+1)$ be given. Let $B$ be a
ball in $\R$ and $h:3^{n+1}B \to \GL_{n+1}(\R)$ be given. Assume that for
any linearly independent collection of vectors $\vv v_1,\dots,\vv v_r\in \Z^{n+1}$
\begin{enumerate}
\item[{\rm(i)}] \hspace*{2ex}the function $x\mapsto \|h(x)\vv v_1\we\dots\we h(x)\vv v_r\|_\infty$ is
$(C,\alpha)$-good\ on $3^{n+1}B$, and\\[-1ex]
\item[{\rm(ii)}] \hspace*{2ex}$\sup\limits_{x\in B}\|h(x)\vv v_1\we\dots\we h(x)\vv v_r\|_\infty\ge\rho$.
\end{enumerate}
Then for any $ \ve >0$
\begin{equation}\label{vb3}
\left|\Big\{x\in B:\delta\big(h(x)\Z^{n+1}\big)\le
\ve\Big\}\right| \hspace*{1ex}\le\hspace*{1ex} (n+1)C 6^{n+1} \left(\frac\ve \rho
\right)^\alpha |B|\,,
\end{equation}
where $\delta(\cdot)$ is given by \eqref{vb0}.
\end{theoremKM}

\bigskip

We  now bring to the forefront the role this theorem  plays in establishing Theorem~\ref{t4}. In the case $d=1$, $m=n-1$ the matrix $G$ given by~\eqref{section_DP_def_gG} has the following form:

\begin{equation} \label{subsection_CAcurve_def_gG}
G(x)=\left(
\begin{array}{cccccccc}
f_1-x f_1'&f_1' & -1&0&\dots& 0\\[1ex]
f_2-x f_2'&f_2' & 0&-1&\dots& 0\\[1ex]
\vdots&\vdots&\vdots&\vdots&\ddots&\vdots\\[1ex]
f_{n-1}-x f_{n-1}'&f_{n-1}' & 0&0&\dots& -1\\[1ex]
x &-1 &  0&0&\dots& 0\\[1ex]
1&0 &  0&0&\dots& 0
\end{array}
\right).
\end{equation}

\noindent We define
\begin{equation}\label{h}
h(x):=D\,g^{-1}G(x)   \quad\text{where } \quad D:=\diag\{c^{1/(n+1)},\dots,c^{1/(n+1)}\}
\end{equation}
and $g=g(c,Q,\psi)$ is as in \eqref{section_DP_def_g}.
Then, by \eqref{def_Bdelta} it follows that

\begin{equation} \label{vb2}
B\setminus \cG(c,Q,\psi)=\Big\{x\in B:\delta\big(D\,g^{-1}G(x)\Z^{n+1}\big)< c^{1/(n+1)}\Big\}\,.
\end{equation}
Therefore, the measure of $B\setminus \cG(c,Q,\psi)$ can be estimated via \eqref{vb3} subject to verifying conditions (i) and (ii) of Theorem~KM. These conditions involve the quantity $\|h(x)\vv v_1\we\dots\we h(x)\vv v_r\|_\infty$ which, by definition, is the maximum of the absolute values of the coordinates of
$h(x)\vv v_1\we\dots\we h(x)\vv v_r$ in the standard basis. The coordinates run over all possible $r\times r$
minors of the matrix $h(x)\Gamma$, where $\Gamma$ is composed of the vectors $\vv v_i$ as its columns.
Hence,
\begin{equation}\label{vb100}
\|h(x)\vv v_1\we\dots\we h(x)\vv v_r\|_\infty= \max_{I=\{i_1,\dots,i_r\}\subset\{1,\dots,n+1\}}\det\big(h_I(x)\Gamma\big),
\end{equation}
where $h_I(x)$ stands for the $r\times(n+1)$ matrix formed by the rows $i_1,\dots,i_r$ of $h(x)$ and $\Gamma=(\vv v_1,\dots,\vv v_r)$ is an $(n+1)\times r$ matrix over $\Z$ of rank $r$.

Throughout the rest of the paper the set of integer $k\times r$ matrices over $\Z$ (respectively, over $\R$) will be denoted by $\Mat_\Z(k,r)$ (respectively, by $\Mat_\R(k,r)$). In turn, the subset of $\Mat_\Z(k,r)$ of full rank, that is of rank $\max\{k,r\}$, will be denoted by $\Mat^*_\Z(k,r)$, and the subset of $\Gamma\in\Mat_\R(k,r)$ with $\|\Gamma\|_\infty\ge1$ will be denoted by $\Mat^*_\R(k,r)$. Observe that $\Mat^*_\Z(k,r)\subset \Mat^*_\R(k,r)$.

Given $I=\{i_1,\dots,i_r\}\subset\{1,\dots,n+1\}$, let $G_I(x)$ denote\label{def_G_I} the $r\times(n+1)$ matrix formed by the rows $i_1,\dots,i_r$ of $G(x)$.
Define the function
$$
\phi_{I,\Gamma}(x):=\det\big(G_I(x)\Gamma\big)\,.
$$

%
%

Since $d=1$, we have that
\begin{equation} \label{section_DP_def_g+}
g=\diag\Big\{\underbrace{\psi,\dots,\psi}_{n-1},(\psi^{n-1} Q)^{-1},c Q\Big\}\,.
\end{equation}
Then, for $h$ given by \eqref{h}, we have that
\begin{equation}\label{vb101}
\det\big(h_I(x)\Gamma\big)=c^{\frac{r}{n+1}}\cdot\Phi_I\cdot\phi_{I,\Gamma}(x)\,,
\end{equation}
where
\begin{equation}\label{vb102}
\Phi_I=\left\{\begin{array}{cl}
                \psi^{-r} &\text{ if $n\not\in I$ and $n+1\not\in I$,}\\[1ex]
                \psi^{n-r} Q &\text{ if $n\in I$ and $n+1\not\in I$,}\\[1ex]
                (c\,\psi^{r-1}Q)^{-1} &\text{ if $n\not\in I$ and $n+1\in I$,}\\[1ex]
                c^{-1}\psi^{n-r+1} &\text{ if $n\in I$ and $n+1\in I$.}
              \end{array}
\right.
\end{equation}
In view of \eqref{vb101} and \eqref{vb102} verifying conditions (i) and (ii) of Theorem~KM for our choice of $h$ is reduced to understanding the functions $\phi_{I,\Gamma}(x)$ for all possible choices of $I$ and $\Gamma$. With this in mind we now state the main assertion regarding $\phi_{I,\Gamma}(x)$.

\medskip

\begin{proposition} \label{proposition_curve_Ca_goodness_2}
Let $\bm f=(f_1,\ldots,f_{n-1})$ be a map of one real variable such that $x\mapsto \vv f(x):=(x,\mv f(x))$ is non-degenerate at some point $x_0\in\R$. Then, there exists a sufficiently small open interval $\cU$ centred at $x_0$, $l\in\N$ and a constant $C>0$ satisfying the following. For any interval $B\subset \cU$ there exists a constant $\rho_B>0$ such that for any $1\le r\le n$ and any $\Gamma\in\Mat^*_\Z(n+1,r)$ the following two properties are satisfied\/$:$
\medskip
\begin{itemize}
\item[{\rm(i)}]
for any $I=\{i_1,\dots,i_r\}\subset\{1,\dots,n+1\}$ we have that
\begin{equation}\label{p1}
|\phi_{I,\Gamma}|\quad\text{is \quad$\left(C,\frac{1}{2l-1}\right)$-good on $3^{n+1}B$,}
\end{equation}
\medskip
\item[{\rm(ii)}] for some $I=\{i_1,\dots,i_r\}\subset\big\{1,\dots,\max\{r,n-1\}\big\}$ we have that
\begin{equation} \label{p2}
\sup_{x\in B}|\phi_{I,\Gamma}(x)|\geq \rho_B.
\end{equation}
\end{itemize}
\end{proposition}

\medskip

As we shall see in \S\ref{prooft4}  below, once armed with Proposition \ref{proposition_curve_Ca_goodness_2} it is not difficult to establish  the desired Theorem~\ref{t4}.

\bigskip

\section{Non-degenerate maps and $(C,\alpha)$-good functions}

In this section we collect together several statements regarding $(C,\alpha)$-good functions that will be required during the course of establishing  Proposition \ref{proposition_curve_Ca_goodness_2}.

\medskip

We begin with the following basic lemma which is a direct consequence of Lemma~3.1 in \cite{Bernik-Kleinbock-Margulis-01:MR1829381} (see also \cite[Lemma~3.1]{Kleinbock-Margulis-98:MR1652916}).

\begin{lemma}\label{l:13}
Let $V\subset\R$ be open and $C,\alpha>0$. If $g_1,\dots,g_m$
are $(C,\alpha)$-good functions on $V$ and
$\lambda_1,\dots,\lambda_m\in\R$, then $\max_{i}|\lambda_ig_i| $ is
a $(C',\alpha')$-good function on $V'$ for every $C'\ge C$, $0<\alpha'\le\alpha$ and open subset $V'\subset V$.
\end{lemma}

The next lemma is a straightforward  consequence of the definition of non-degeneracy.

\begin{lemma}\label{lnd}
Let $\bm f=(f_1,\ldots,f_{n-1})$ be a map of one real variable such that $x\mapsto \vv f(x):=(x,\mv f(x))$ is $l$-non-degenerate at some point $x_0\in\R$.
Then for any $r$ indices $1\le i_1<\dots<i_r\le n-1$ the map $x\mapsto (x,f_{i_1}(x),\dots,f_{i_r}(x))$ is $l$-non-degenerate at $x_0$.
\end{lemma}

\bigskip

We will be interested in three particular classes of functions associated with the map
\begin{equation} \label{def_vvf}
\vv f(x):=(x,\mv f(x))=(x,f_1(x),\dots,f_{n-1}(x))\,.
\end{equation}
The first two are
\begin{equation}\label{cF}
  \cF:=\{u_0+\vv u\cdot\vv f(x):\sum_{j=0}^nu_j^2=1\}
\end{equation}
and
\begin{equation}\label{cF'}
  \cF':=\{\vv u\cdot\vv f'(x):\sum_{j=1}^nu_j^2=1\}\,,
\end{equation}
where
$
\vv u=(u_1,\dots,u_n)
$
and the `dot' represents  the standard inner product.  For these two classes we will will make use of the following statement.

\begin{proposition}
\label{prop5.1}
Suppose that the map $x\mapsto\vv f(x)$ is $l$-non-degenerate at some point $x_0 \in \R$. Then there exists a constant $C>0$ and a neighbourhood $V$ of $x_0$ such that
\begin{itemize}
  \item[{\rm(a)}]
every function in $\cF$ is $(C,\tfrac1l)$-good on $V$;\\[0ex]
  \item[{\rm(b)}]
every function in $\cF'$ is $(C,\frac1{l-1})$-good on $V$;\\[0ex]
  \item[{\rm(c)}]
for any interval $B\subset V$ there exist a constant $\rho_B>0$ such that
\begin{equation}\label{f1}
\inf_{f\in\cF}\sup_{x\in B}|f(x)|\ge \rho_B\qquad\text{and}\qquad \inf_{f\in\cF'}\sup_{x\in B}|f(x)|\ge \rho_B\,.
\end{equation}
\end{itemize}
\end{proposition}

\begin{proof}
Parts~(a) and (b) appear as Corollary~3.5 in \cite{Bernik-Kleinbock-Margulis-01:MR1829381};  see also \cite[Proposition~3.4]{Kleinbock-Margulis-98:MR1652916}. For part (c) we choose $V$ sufficiently small so that $\vv f$ is non-degenerate everywhere on $V$. Then, note that the map $(u_0,\vv u)\mapsto\sup_{x\in B}|u_0+\vv u\cdot\vv f(x)|$ is continuous and strictly positive. The latter is due to the linear independence of $1,x,f_1(x),\dots,f_{n-1}(x)$ over $\R$ which in turn is a consequence of the non-degeneracy of $\vv f$ on $B\subset V$. Then
$$
\inf_{(u_0,\vv u)\in\R_1^{n+1}}\ \sup_{x\in B}|u_0+\vv u\cdot\vv f(x)|:=\rho_B>0\,,
$$
since we are taking the infimum of a positive continuous function over a compact set, namely $\R_1^{n+1}$, which is the unit (Euclidean) sphere in $\R^{n+1}$. This proves the first of the inequalities associated with   \eqref{f1}. The proof of the second is similar once we make the observations that the non-degeneracy of $\vv f$ at $x_0$ implies the non-degeneracy of $\mv f'=(f_1',\dots,f_{n-1}')$ at $x_0$.
\end{proof}

\medskip

In what follows, given a map $\vv g=(g_1,g_2)$ of one real variable, the associated function
\begin{equation} \label{def_skew_gradient}
\tilde\nabla\vv g:=g_1g'_2-g_1'g_2
\end{equation}
will be referred to as  the {\em  skew gradient of $\vv g$}. This notion was introduced in~\cite[\S4]{Bernik-Kleinbock-Margulis-01:MR1829381} and the  following statement concerning  the skew gradient is a simplified version of Proposition~4.1 from \cite{Bernik-Kleinbock-Margulis-01:MR1829381}.

\begin{proposition}\label{prop4.2}
Let $U$ be an open interal, $x_0\in U$ and let $\cG$ be a family of $C^l$ maps $\vv g:U\to\R^2$ such that
\begin{equation}\label{c1}
\text{the family }\big\{g'_i:\vv g=(g_1,g_2)\in\cG,~i=1,2\big\}\text{ is compact in }C^{l-1}(U).
\end{equation}
Assume also that
\begin{equation}\label{c2}
\inf_{\vv v\in\R_1^2}\,\,\inf_{\vv g\in\cG}\,\,\max_{1\le i\le l}|\vv v\cdot\vv g^{(i)}(x_0)|>0\,.
\end{equation}
Then there exists a constant $C>0$ and a neighbourhood $V$ of $x_0$ such that
\begin{itemize}
  \item[{\rm(a)}] $|\tilde\nabla\vv g|$ is $(C,\frac{1}{2l-1})$-good on $V$ for all $\vv g\in\cG$,\\
  \item[{\rm(b)}] for every interval $B\subset V$ there exists $\rho_B>0$ such that for all  $\vv g\in\cG$
  $$
  \sup_{x\in B}|\tilde\nabla\vv g(x)|\ge\rho_B\,.
  $$
\end{itemize}
\end{proposition}

The third class of functions that we will be interested in, associated with the map  $ \vv f$ defined  by
\eqref{def_vvf},   is the class

\begin{equation}\label{cG}
\cG:=\left\{\big(\vv u_1\cdot\vv f(x),u_0+\vv u_2\cdot\vv f(x)\big):
\begin{array}{l}
\vv u_1,\vv u_2\in\R^n_1,\\
\vv u_1\cdot\vv u_2=0
\end{array}
\right\}\,.
\end{equation}

\medskip

\noindent In particular, we will make use of the following statement concerning the skew gradient of maps in $\cG$.  It follows on showing that Proposition \ref{prop4.2} is applicable to the specific  $\cG$ given by \eqref{cG}.

\begin{proposition}\label{prop5.3}
Suppose that the map $x\mapsto\vv f(x)$ is $l$-non-degenerate at some point $x_0$ and $\cG$ is given by \eqref{cG}. Then there exists a constant $C>0$ and a neighbourhood $V$ of $x_0$ such that
\begin{itemize}
  \item[{\rm(a)}]
for every $\vv g\in\cG$ the function $|\tilde\nabla\vv g|$ is $(C,\frac{1}{2l-1})$-good on $V$;\\[0ex]
  \item[{\rm(b)}]
for any interval $B\subset V$ there exist a constant $\rho_B>0$ such that
\begin{equation}\label{f1+++}
\inf_{\vv g\in\cG}~\sup_{x\in B}|\tilde\nabla \vv g(x)|\ge \rho_B\,.
\end{equation}
\end{itemize}
\end{proposition}

\begin{proof}
Since $\vv f$ is $l$-non-degenerate at $x_0$, there exists an open interval $U$ centred at $x_0$ such that $\vv f$ is $C^l$ on the closure of $U$. Hence the family of functions within \eqref{c1}, which is simply $\{\vv u\cdot\vv f':\vv u\in\R^n_1\}$, is compact in $C^{l-1}(U)$ due to the compactness of the unit sphere in $\R^{n}$. Thus the first hypothesis \eqref{c1} is satisfied.

\noindent Next, given $\vv g=(\vv u_1\cdot\vv f,u_0+\vv u_2\cdot\vv f)\in\cG$ and $\vv v=(v_1,v_2)\in\R^2_1$, we have that
$$
\vv v\cdot\vv g=u_0'+\vv u'\cdot\vv f\,,
$$
where
$$
(u_0',\vv u')=v_1(0,\vv u_1)+v_2(u_0,\vv u_2)\,.
$$
Since the vectors $\vv u_1$ and $\vv u_2$ are orthonormal, we have that
$$
\|\vv u'\|_2^2 \, =  \, v_1^2\|\vv u_1\|_2^2+v_2^2\|\vv u_2\|_2^2 \, = \, v_1^2+v_2^2=1\,
$$
and thus \eqref{c2} immediately follows from the $l$-non-degeneracy of $\vv f(x)$ at $x_0$.

The upshot of the above is that the desired statements (a) and (b) of Proposition~\ref{prop5.3} now directly follow on applying  Proposition~\ref{prop4.2}.
\end{proof}

\section{Proof of Proposition~\ref{proposition_curve_Ca_goodness_2}}\label{section_Ca_goodness}

The proof of Proposition~\ref{proposition_curve_Ca_goodness_2} is split into several lemmas.
We will use various properties of multi-vectors and their relations with linear subspaces, which can be found in \cite[\S3]{Beresnevich-SDA1}, \cite{Schmidt-1980} and \cite{Schmidt2}.
We begin with an auxiliary statement that will be helpful for calculating  $\phi_{I,\Gamma}(x)$.

\begin{lemma}\label{hodge}
Suppose that $\Gamma= (\vv v_1,\dots,\vv v_r) \in\Mat_\R^*(n+1,r)$ with $1\le r\le n$ and $\vv v_1,\dots,\vv v_r$ denote the columns of\/ $\Gamma$. Let $I=\{i_1,\dots,i_r\}\subset\{1,\dots,n+1\}$ and $G_{i_1}(x),\ldots,G_{i_r}(x)$ be the corresponding rows of $G(x)$. Then for any collection $\vv a_1,\dots,\vv a_{n+1-r}\in\R^{n+1}$ of linearly independent rows such that $\vv a_i\vv v_j=0$ for all $i=1,\dots,n+1-r$, $j=1,\dots,r$ and
\begin{equation}\label{v+}
\|\vv a_1\we\dots\we\vv a_{n+1-r}\|=\|\vv v_1\we\dots\we\vv v_{r}\|
\end{equation}
we have that
\begin{align}
\label{f13}|\phi_{I,\Gamma}(x)|&=\|G_{i_1}(x)\we\dots\we G_{i_r}(x)\we\vv a_1\we\dots\we\vv a_{n+1-r}\|\\[1ex]
\nonumber&=|\det\big(G_{i_1}(x),\dots,G_{i_r}(x),\vv a_1,\dots,\vv a_{n+1-r}\big)|\,.
\end{align}
Furthermore, $\vv a_1,\dots,\vv a_{n+1-r}$ can be taken to be integer if $\Gamma\in\Mat_\Z^*(n+1,r)$. \end{lemma}

\begin{proof}
Let
\begin{equation} \label{def_w}
\vv w:=\vv v_1\we\dots\we\vv v_{r}
\end{equation}
and
\begin{equation} \label{def_G}
{\vv m}(x):=G_{i_1}(x)\wedge\ldots\wedge G_{i_r}(x).
\end{equation}
Then, by the Laplace identity (see, for example, \cite[Equation~(3.3)]{Beresnevich-SDA1} or \cite[Lemma~5E]{Schmidt2}) it follows that
$$
\phi_{I,\Gamma}(x)={\vv m}(x)\cdot{\vv w}\,,
$$
where the `dot' represents  the standard inner product on $\bigwedge^{n-1}(\R^{n+1})$.
Let $\vv w^\perp$ be the Hodge dual of $\vv w$, see \cite[\S3]{Beresnevich-SDA1} for its definition and properties. Since $\vv w$ is decomposable, so is ${\vv w}^{\bot}$. This means that
\begin{equation}\label{f13B}
{\vv w}^{\bot}=\vv a_1\wedge\dots\we\vv a_{n+1-r}
\end{equation}
for some linearly independent rows $\vv a_1,\dots,\vv a_{n+1-r}$ which form a basis
of the linear subspace of $\R^{n+1}$ orthogonal to $\vv v_1,\dots,\vv v_{r}$. Equation \eqref{v+} is a consequence of the Hodge operator being an isometry, see \cite[\S3.2]{Beresnevich-SDA1}.
Furthermore, it follows from~\cite[Lemma~5G]{Schmidt2} that it is possible to choose $\vv a_i$, for all $i=1,\dots,n+1-r$, to be integer vectors in the case $\Gamma$ is an integer matrix.

By  duality (see \cite[Equation~(3.10)]{Beresnevich-SDA1}), we have that
\begin{align*}
|\phi_{I,\Gamma}(x)|&=|{\vv m}(x)\cdot{\vv w}|=|{\vv m}(x)\wedge{\vv w}^{\bot}|
\end{align*}
whence \eqref{f13} follows on substituting \eqref{def_G} and \eqref{f13B}.
\end{proof}

\medskip

\begin{lemma} \label{lemma_curve_case_one}
With reference to Proposition~\ref{proposition_curve_Ca_goodness_2}
\begin{itemize}
\item
Statement \eqref{p1} holds if $r=n-1$ and $I=\{1,\dots,n-1\}$,\\[-1ex]

\item
Statement \eqref{p2} holds if $r=n-1$ and
 $I$ as above.
\end{itemize}
\end{lemma}

\begin{proof}
Although, in the context of Proposition~\ref{proposition_curve_Ca_goodness_2} we are ultimately interested in integer $\Gamma$, it will be necessary for the proof of Lemma~\ref{lemma_curve_case_one} to consider $\Gamma$ lying in the larger set $\Mat^*_\R(n+1,r)$.  With this in mind,
by Lemma~\ref{hodge} there exist linearly independent row-vectors $\vv a_1,\vv a_2\in\R^{n+1}$ such that
$$
|\phi_{I,\Gamma}(x)|=|G_1(x)\wedge\ldots\wedge G_{n-1}(x)\we\vv a_1\we\vv a_2|.
$$
For convenience, and in view of Lemma~\ref{hodge} without loss of generality, we take
$$
\vv a_1=(0,\vv u_1)\qquad\text{and}\qquad\vv a_2=(u_0,\vv u_2)
$$
such that $\vv u_1\cdot\vv u_2=0$, $\vv u_1\in\R^n_1$ and $\vv u_2\in\R^n$,
where $\vv u_i=(u_{i,1},\dots,u_{i,n})$ for $i=1,2$.
Thus, $|\phi_{I,\Gamma}(x)|$ is equal to the absolute value of the determinant of the following $(n+1)\times (n+1)$ matrix:
\begin{equation} \label{subsection_CAcurve_def_det_psi}
\Psi_{x}:=\left(
\begin{array}{cccccccc}
f_1-x f_1'&f_1' & -1&0&\dots& 0\\[1ex]
f_2-x f_2'&f_2' & 0&-1&\dots& 0\\[1ex]
\vdots&\vdots&\vdots&\vdots&\ddots&\vdots\\[1ex]
f_{n-1}-x f_{n-1}'&f_{n-1}' & 0&0&\dots& -1\\[1ex]
0 & u_{1,1} &  u_{1,2} &  u_{1,3} & \dots & u_{1,n}\\[1ex]
u_{0} & u_{2,1} &  u_{2,2} &  u_{2,3} & \dots & u_{2,n}
\end{array}
\right)\,,
\end{equation}

\noindent To proceed, define the following auxiliary $(n+1)\times(n+1)$ matrix

\begin{equation} \label{subsection_CAcurve_def_det_additional_matrix}
\xi_{x}:=\left(
\begin{array}{cccccc}
1 & 0 & 0 & 0 & \dots & 0\\[0ex]
x & 1 & 0 & 0 & \dots&  0\\[0ex]
f_1 & f_1' & 1 & 0 & \dots&  0\\[0ex]
f_2 & f_2' & 0 & 1 & \dots&  0\\[0ex]
\vdots&\vdots&\vdots&\vdots&\ddots&\vdots\\[0ex]
f_{n-1} & f_{n-1}' & 0 & 0 & \dots&  1
\end{array}
\right)
\end{equation}

\noindent Since $\det\xi_x=1$, we have that
\begin{equation} \label{subsection_CAcurve_triple}
|\phi_{I,\Gamma}(x)|=|\det\Psi_x|=|\det\left(\Psi_x\xi_x\right)|.
\end{equation}

\noindent On the other hand, we have that

\begin{equation} \label{subsection_CAcurve_psi_xi}
\Psi_{x}\xi_x=\left(
\begin{array}{cccccccc}
0 & 0 & -1&0&\dots& 0\\[1ex]
0 & 0 & 0&-1&\dots& 0\\[1ex]
\vdots&\vdots&\vdots&\vdots&\ddots&\vdots\\[1ex]
0 & 0 & 0&0&\dots& -1\\[1ex]
\vv u_1\cdot\vv f(x) & \vv u_1\cdot\vv f'(x)     &  u_{1,2} &  u_{1,3} & \dots & u_{1,n}\\[2ex]
u_0+\vv u_2\cdot\vv f(x) & \vv u_2\cdot\vv f'(x) &  u_{2,2} &  u_{2,3} & \dots & u_{2,n}
\end{array}
\right)\,,
\end{equation}
where as usual $\vv f$ is given by~\eqref{def_vvf}.

If
$$
\vv u_2=\vv0
$$
then $u_0\neq0$ and
\begin{equation}\label{v90}
|\phi_{I,\Gamma}(x)|=|u_0\vv u_1\cdot\vv f'(x)|,
\end{equation}
which is a non-zero multiple of the absolute value of a function from the class $\cF'$ defined by~\eqref{cF'}. If $\vv u_2\neq\vv0$, then we have that
\begin{align} \label{subsection_CAcurve_psi_xi_skew_gradient}
|\phi_{I,\Gamma}(x)|&=\big|\tilde{\nabla} \big(\vv u_1\cdot\vv f(x),u_0+\lambda\vv u_2\cdot\vv f(x)\big)\big|\\[1ex]
&\nonumber=\|\vv u_1\|_2\|\vv u_2\|_2~\big|\tilde{\nabla} \big(\|\vv u_1\|_2^{-1}\vv u_1\cdot\vv f(x),\|\vv u_2\|_2^{-1}u_0+\|\vv u_2\|_2^{-1}\vv u_2\cdot\vv f(x)\big)\big|
\end{align}
which is a non-zero multiple of the absolute value of a function from the class $\cG$ defined by~\eqref{cG}. Therefore, on combining Propositions~\ref{prop5.1} and \ref{prop5.3} together with Lemma~\ref{l:13} we conclude that $|\phi_{I,\Gamma}|$ is $(C,\frac{1}{2l-1})$-good on $V$ for a suitably chosen constant $C>0$ and neighbourhood $V$ of $x_0$.  This completes the proof of the first part of the lemma.

Since $\|\Gamma\|_\infty\ge1$, the vector $\vv w$ defined by~\eqref{def_w} satisfies $\|\vv w\|_2\ge1$. By Propositions~\ref{prop5.1} and \ref{prop5.3}, it follows that for any ball $B\subset V$
\begin{equation}\label{f12}
\sup_{x\in B}|\phi_{I,\Gamma}(x)|=\sup_{x\in B}|{\vv m}(x)\cdot{\vv w}|= \|\vv w\|_2\sup_{x\in B}|{\vv m}(x)\cdot{\vv w'}|>0\,,
\end{equation}
where $\vv w'=\vv w/\|\vv w\|_2$ is a unit decomposable multivector. Note that  the set of decomposable unit multivectors $\vv w'\in\bigwedge^{2}(\R^{n+1})$ is compact and $\vv w\mapsto \sup_{x\in B}|{\vv m}(x)\cdot{\vv w}|$ is strictly positive and continuous (see \cite[p.\,218]{Beresnevich-SDA1}). Then taking the infimum in \eqref{f12} over $\vv w'$ implies that the right hand side of \eqref{f12} is bounded away from zero by a constant $\rho_B>0$. This completes the proof of the second part of the lemma.
\end{proof}

\medskip

\begin{lemma}\label{corollary_curve_case_one}
\samepage
With reference to Proposition~\ref{proposition_curve_Ca_goodness_2}
\begin{itemize}
  \item
  Statement \eqref{p1} holds if $r\le n-1$ and $I\subset\{1,\dots,n-1\}$, \\[-1ex]
  \item
  Statement \eqref{p2} holds if $r\le n-1$ for some $I\subset\{1,\dots,n-1\}$.
\end{itemize}
\end{lemma}

\begin{proof}
With in the context of  Proposition~\ref{proposition_curve_Ca_goodness_2}, we are given that $\Gamma\in\Mat^*_\Z(n+1,r)$.  The fact that $\Gamma$ is an integer is absolutely  crucial in the proof of Lemma~\ref{corollary_curve_case_one}. In short, it allows us to make an induction step to a lower dimension statement.

Fix any multiindex $I$ such that $n\not\in I$ and $n+1\not\in I$. Recall that the matrix $G(x)$ is defined by~\eqref{subsection_CAcurve_def_gG}. Consider the auxiliary $(r+2)\times (r+2)$ matrix $\tilde{G}(x)$ formed by the rows $i_1,\dots,i_r,n,n+1$ and columns $1,2,i_1+2,\dots,i_r+2$ of $G(x)$. Also, consider the matrix $\tilde{\Gamma}$ formed by the rows $1,2,i_1+2,\dots,i_r+2$ of the matrix $\Gamma$.
Observe that
\begin{equation}\label{vb6789}
\phi_{I,\Gamma}(x)=\det\tilde G_{\tilde I}(x)\tilde\Gamma\,,
\end{equation}
where $\tilde I:=\{1,\dots,r\}$. This is because when going from $G_I$ to $\widetilde G_{\widetilde I}$ we simply cross out zero columns and so $G_I(x)\Gamma=\widetilde G_{\widetilde I}(x)\widetilde \Gamma$. Thus the  desired properties of $\phi_{I,\Gamma}(x)$ can be investigated via the lower dimensional matrix $\tilde G$, which has exactly the same structure as $G$. Indeed the matrix $\tilde G$ is the analogue of $G$ with the associated map $\vv f$ replaced by
\begin{equation} \label{f_i}
\tilde{\vv f} (x)    \to (x,f_{i_1}(x),\dots,f_{i_r}(x))  \, .
\end{equation}
 Note that by Lemma~\ref{lnd}, the map $\tilde{\vv f}$ is $l$-non-degenerate at $x_0$.  With this in mind and without loss of generality assuming that $\phi_{I,\Gamma}$  is given by \eqref{vb6789},  we are in the position to prove the lemma.

 Suppose to start with that $\tilde\Gamma$ has rank $<r$.  Then,  we have that  $\phi_{I,\Gamma}$ is identically zero and it follows that  $|\phi_{I,\Gamma}|$ is $(C,\alpha)$-good for any choice of $C$ and $\alpha$.   Now suppose that $\tilde\Gamma$ has rank exactly $r$,  that is $\tilde{\Gamma}\in\Mat^*_\Z(r+2,r)$. Then,  on applying Lemma~\ref{lemma_curve_case_one} with $\tilde{\vv f} $ in place of $\vv f$, $\tilde\Gamma$ in place of $\Gamma$, $r+1$ in place of $n$, $\tilde{G}(x)$ in place of $G(x)$ and $\tilde I$ in place of $I$, we complete  the proof of the  first part of  Lemma~\ref{corollary_curve_case_one}. To verify the second part it remains to note that, since $\rank\Gamma=r$, there is always a choice of $I=\{i_1<\dots<i_r\}\subset\{1,\dots,n-1\}$ such that $\tilde\Gamma$ defined above has rank $r$.
\end{proof}

\medskip

\begin{lemma} \label{lemma_curve_case_two}
With reference to Proposition~\ref{proposition_curve_Ca_goodness_2}
\begin{itemize}
  \item
  Statement \eqref{p1} holds if $r=n$ and $I=\{1,\dots,n\}$, \\[-1ex]
  \item
  Statement \eqref{p2} holds if $r=n$ and $I$ as above.
\end{itemize}
\end{lemma}

\begin{proof}
By Lemma~\ref{hodge}, $|\phi_{I,\Gamma}(x)|=|\det\Psi_x|$, where

\begin{equation} \label{lemma_curve_case_two_def_psi}
\Psi_x:=\left(
\begin{array}{cccccccc}
f_1-x f_1'&f_1' & -1&0&\dots& 0\\[1ex]
f_2-x f_2'&f_2' & 0&-1&\dots& 0\\[1ex]
\vdots&\vdots&\vdots&\vdots&\ddots&\vdots\\[1ex]
f_{n-1}-x f_{n-1}'&f_{n-1}' & 0&0&\dots& -1\\[1ex]
x & -1 & 0 & 0 & \dots & 0\\[1ex]
a_0 & a_1 & a_2 & a_3 & \dots & a_{n}
\end{array}
\right)
\end{equation}
for some non-zero integer vector $(a_0,\dots,a_{n})$. With $\xi_x$ given by~\eqref{subsection_CAcurve_def_det_additional_matrix}, we obtain that

$$
\Psi_x\xi_x=\left(
\begin{array}{cccccccc}
0 & 0 & -1&0&\dots& 0\\[1ex]
0 & 0 & 0&-1&\dots& 0\\[1ex]
\vdots&\vdots&\vdots&\vdots&\ddots&\vdots\\[1ex]
0 & 0 & 0&0&\dots& -1\\[1ex]
0 & -1 &  0 &  0 & \dots & 0\\[1ex]
a_{0}+xa_{1}+\sum_{j=1}^{n-1}a_{j+1}f_j(x) & a_{1}+\sum_{j=1}^{n-1}a_{j+1}f_j' &  0 &  0 & \dots & 0
\end{array} \, .
\right)\, .
$$
Hence
$$
|\phi_{I,\Gamma}(x)|=\left|\det\Psi_x\xi_x\right|=|a_{0}+xa_{1}+\sum_{j=1}^{n-1}a_{j+1}f_j(x)|
$$
is a constant multiple of a function from  the  class $\cF$ defined by \eqref{cF}. Since $(a_0,\dots,a_n)$ is a non-zero integer vector, the constant multiple in question is $\ge1$. Therefore, the  lemma readily follows form Proposition~\ref{prop5.1} (parts (a) and (c))  together  with Lemma~\ref{l:13}.
\end{proof}

\begin{lemma}\label{corollary_curve_case_two}
With reference to Proposition~\ref{proposition_curve_Ca_goodness_2}
\begin{itemize}
  \item
  Statement \eqref{p1} holds if $r\le n$, $n\in I$ and $n+1\not \in I$.
\end{itemize}
\end{lemma}

\begin{proof}
To start with observe that when  $r=1$, we necessarily have that $I=\{n\}$ and  thus $G_I(x)\Gamma$ is either identically zero or is a non-zero linear function. In the former case  it easily follows  that $|\phi_{I,\Gamma}|$ is $(C,\alpha)$-good  for any $C$ and $\alpha$.  In the latter case, $\phi_{I,\Gamma} $ it is a multiple of an element of the class $\cF$ defined by~\eqref{cF} and so by Lemma~\ref{l:13}, $|\phi_{I,\Gamma}|$ is $(C,\frac1l)$-good on some neighbourhood of $x_0$.

Without loss of generality,  we assume that $ r > 1$.
Fix any multiindex $I$ such that $n\in I$ and $n+1\not\in I$. Recall that the matrix $G(x)$ is defined by~\eqref{subsection_CAcurve_def_gG}. Consider the auxiliary $(r+1)\times (r+1)$ matrix $\tilde{G}(x)$ formed by the rows $i_1,\dots,i_r,n+1$ and columns $1,2,i_1+2,\dots,i_{r-1}+2$ of $G(x)$. Note that $i_r=n$. Also, consider the matrix $\tilde{\Gamma}$  formed by the rows $1,2,i_1+2,\dots,i_{r-1}+2$ of the matrix $\Gamma$.
Observe that
\begin{equation}\label{vb6789B}
\phi_{I,\Gamma}(x)=\det\tilde G_{\tilde I}(x)\tilde\Gamma\,,
\end{equation}
where $\tilde I=\{1,\dots,r\}$.  Thus the  desired properties of $\phi_{I,\Gamma}(x)$ can be investigated via the lower dimensional matrix $\tilde G$, which has exactly the same structure as $G$. Indeed the matrix $\tilde G$ is the analogue of $G$ with the associated $\vv f$ replaced by
\begin{equation} \label{f_iB}
\tilde{\vv f} (x)    \to  (x,f_{i_1}(x),\dots,f_{i_{r-1}}(x))  \, .
\end{equation}
Note that by Lemma~\ref{lnd}, the map $\tilde{\vv f}$ is $l$-non-degenerate at $x_0$.  With this in mind and without loss of generality assuming that $\phi_{I,\Gamma}$  is given by \eqref{vb6789B},  we are in the position to prove the lemma.

 Suppose to start with that $\tilde\Gamma$ has rank $<r$.  Then, we have that  $\phi_{I,\Gamma}$ is identically zero and it follows that  $|\phi_{I,\Gamma}|$ is $(C,\alpha)$-good for any choice of $C$ and $\alpha$.   Now suppose that $\tilde\Gamma$ has rank exactly $r$,  that is $\tilde{\Gamma}\in\Mat^*_\Z(r+1,r)$. Then,  on applying Lemma~\ref{lemma_curve_case_two} with $\tilde{\vv f} $ in place of $\vv f$, $\tilde\Gamma$ in place of $\Gamma$, $r+1$ in place of $n$, $\tilde{G}(x)$ in place of $G(x)$ and $\tilde I$ in place of $I$, completes the proof of  Lemma~\ref{corollary_curve_case_two}.
\end{proof}

\medskip

\begin{lemma} \label{lemma_curve_case_xhree}
With reference to Proposition~\ref{proposition_curve_Ca_goodness_2}
\begin{itemize}
  \item
  Statement \eqref{p1} holds if $r=n$ and $I=\{1,\dots,n-1,n+1\}$.
\end{itemize}
\end{lemma}

\medskip

\begin{proof}
By Lemma~\ref{hodge}, $|\phi_{I,\Gamma}(x)|=|\det\Psi_x|$, where

\begin{equation} \label{lemma_curve_case_xhree_def_psi}
\Psi_x:=\left(
\begin{array}{cccccccc}
f_1-x f_1'&f_1' & -1&0&\dots& 0\\[1ex]
f_2-x f_2'&f_2' & 0&-1&\dots& 0\\[1ex]
\vdots&\vdots&\vdots&\vdots&\ddots&\vdots\\[1ex]
f_m-x f_m'&f_m' & 0&0&\dots& -1\\[1ex]
1 & 0 & 0 & 0 & \dots & 0\\[1ex]
a_0 & a_1 & a_2 & a_3 & \dots & a_{n}
\end{array}
\right)
\end{equation}

\noindent for some non-zero integer vector $(a_0,\dots,a_{n})$. With $\xi_x$ given by~\eqref{subsection_CAcurve_def_det_additional_matrix}, we obtain that

$$
\Psi_x\xi_x=\left(
\begin{array}{cccccccc}
0 & 0 & -1&0&\dots& 0\\[1ex]
0 & 0 & 0&-1&\dots& 0\\[1ex]
\vdots&\vdots&\vdots&\vdots&\ddots&\vdots\\[1ex]
0 & 0 & 0&0&\dots& -1\\[1ex]
1 & 0 &  0 &  0 & \dots & 0\\[1ex]
a_{0}+xa_{1}+\sum_{k=1}^{n-1}a_{k+1}f_k & a_{1}+\sum_{k=1}^{n-1}a_{k+1}f_k' &  0 &  0 & \dots & 0
\end{array}
\right).
$$
Hence
$$
|\phi_{I,\Gamma}(x)|=\left|\det\Psi_x\xi_x\right|=|a_{1}+\sum_{k=1}^{n-1}a_{k+1}f_k'(x)|
$$
is either identically zero or a constant multiple of a function from $\cF'$. In the latter case the claim of the lemma readily follow form Proposition~\ref{prop5.1}(b) combined with Lemma~\ref{l:13}. In the case $|\phi_{I,\Gamma}(x)|$ is identically zero the claim is trivial. Indeed, in that case $|\phi_{I,\Gamma}|$ is $(C,\alpha)$-good for any choice of $C$ and $\alpha$.
\end{proof}

\medskip

\begin{lemma}\label{corollary_curve_case_xhree}
With reference to Proposition~\ref{proposition_curve_Ca_goodness_2}
\begin{itemize}
  \item
  Statement \eqref{p1} holds if $r\le n$, $n\not\in I$ and $n+1\in I$.
\end{itemize}
\end{lemma}

\medskip

\begin{proof}

To start with observe that when  $r=1$, then we necessarily have that $I=\{n+1\}$ and  thus $G_I(x)\Gamma$ is either identically zero or is a non-zero constant. In the former case  it easily follows  that $|\phi_{I,\Gamma}|$ is $(C,\alpha)$-good  for any $C$ and $\alpha$.  In the latter case, $\phi_{I,\Gamma} $ it is a multiple of an element of the class $\cF$ defined by~\eqref{cF} and so by Lemma~\ref{l:13}, $|\phi_{I,\Gamma}|$ is $(C,\frac1l)$-good on some neighbourhood of $x_0$.

Without loss of generality, we assume that $ r > 1$.   Fix any multiindex $I$ such that $n\not\in I$ and $n+1\in I$. Recall that the matrix $G(x)$ is defined by~\eqref{subsection_CAcurve_def_gG}. Consider the auxiliary $(r+1)\times (r+1)$ matrix $\tilde{G}(x)$ formed by the rows $i_1,\dots,i_{r-1},n,n+1$ and columns $1,2,i_1+2,\dots,i_{r-1}+2$ of $G(x)$. Note that $i_r=n+1$. Also, consider the matrix $\tilde{\Gamma}$  formed by the rows $1,2,i_1+2,\dots,i_{r-1}+2$ of the matrix $\Gamma$.
Observe that
\begin{equation}\label{vb6789C}
\phi_{I,\Gamma}(x)=\det\tilde G_{\tilde I}(x)\tilde\Gamma\,,
\end{equation}
where $\tilde I=\{1,\dots,r\}$.  Thus the  desired properties of $\phi_{I,\Gamma}(x)$ can be investigated via the lower dimensional matrix $\tilde G$, which has exactly the same structure as $G$. Indeed the matrix $\tilde G$ is the analogue of $G$ with the associated $\vv f$ replaced by
\begin{equation} \label{f_iC}
\tilde{\vv f} (x)    \to  (x,f_{i_1}(x),\dots,f_{i_{r-1}}(x))  \, .
\end{equation}
Note that by Lemma~\ref{lnd}, the map $\tilde{\vv f}$ is $l$-non-degenerate at $x_0$.  With this in mind and without loss of generality assuming that $\phi_{I,\Gamma}$  is given by \eqref{vb6789C},  we are in the position to prove the lemma.

Suppose to start with that $\tilde\Gamma$ has rank $<r$.  Then, we have that  $\phi_{I,\Gamma}$ is identically zero and it follows that  $|\phi_{I,\Gamma}|$ is $(C,\alpha)$-good for any choice of $C$ and $\alpha$.   Now suppose that $\tilde\Gamma$ has rank exactly $r$,  that is $\tilde{\Gamma}\in\Mat^*_\Z(r+1,r)$. Then,  on applying Lemma~\ref{lemma_curve_case_xhree} with $\tilde{\vv f} $ in place of $\vv f$, $\tilde\Gamma$ in place of $\Gamma$, $r+1$ in place of $n$, $\tilde{G}(x)$ in place of $G(x)$ and $\tilde I$ in place of $I$, completes the proof of Lemma~\ref{corollary_curve_case_xhree}.

\end{proof}

\medskip

\begin{proof}[Completion of the proof of Proposition~\ref{proposition_curve_Ca_goodness_2}.] First of all note that Property~\eqref{p1} has already been established in Lemma~\ref{corollary_curve_case_one} if $I\cap\{n,n+1\}=\emptyset$, in
Lemma~\ref{corollary_curve_case_two} if $I\cap\{n,n+1\}=\{n\}$ and in
Lemma~\ref{corollary_curve_case_xhree} if $I\cap\{n,n+1\}=\{n+1\}$.
If $I\cap\{n,n+1\}=\{n,n+1\}$, then $G_I(x)\Gamma$ is readily seen to be constant, which is thus $(C,1/l)$-good. Therefore, for any $1\le r\le n$ and any $\Gamma\in\Mat^*_\Z(n+1,r)$ Property~(i) holds for all choices of $I$.

Regarding Property~{\rm(ii)} of Proposition~\ref{proposition_curve_Ca_goodness_2},
if $r=n$ then it is established in Lemma~\ref{lemma_curve_case_two} and  if $r<n$ it is established in Lemma~\ref{corollary_curve_case_one}.
\end{proof}

\begin{remark}\rm
The constants $C$ and $\rho_B$ that arise from the various lemmas proved in this section may in principle depend on the choice $I$. However, since there are only finitely many different choices of $I$ both the constants in question can  be made independent of I.  Indeed $\rho_B$ has to be taken as  the minimum while $C$ has to be taken to be the maximum of all the possible values over all different choices of $I$. The fact that the maximum  choice for $C$ works for all $I$  is a consequence of  Lemma~\ref{l:13}.
\end{remark}

\section{Proof of Theorem~\ref{t4}}\label{prooft4}

Let $h$ be given by \eqref{h}  and let $x_0 \in \R$ be such that $\vv f$ is non-degenerate at $x_0$. Then, by \eqref{vb100}, \eqref{vb101}, Proposition~\ref{proposition_curve_Ca_goodness_2} and Lemma~\ref{l:13}, there exists a neighbourhood $\cU$ of $x_0$ such that for any collection of linearly independent integer points $\vv v_1,\dots,\vv v_r$ $(1\le r\le n)$ the map

$$x\mapsto \|h(x)\vv v_1\we\dots\we h(x)\vv v_r\|_\infty \mbox{ is $(C,\frac{1}{2l-1})$-good on $3^{n+1}\cU$  } $$
 and
$$
\sup_{x\in B}\|h(x)\vv v_1\we\dots\we h(x)\vv v_r\|_\infty\,\ge\,
c^{\frac{r}{n+1}}\,\rho_{B}
\min_I \Phi_I\,,
$$
where the minimum is taken over  $ I=\{i_1,\dots,i_r\}\subset\big\{1,\dots,\max\{r,n-1\}\big\}$ and    $  \Phi_I $ is given by  \eqref{vb102}.    It follows from the definition of $  \Phi_I $, that for $r\le n$
$$
\min_{I}\Phi_I  \, \ge  \,  \min\{\psi^{-r},Q\}
$$
and consequently
$$
\sup_{x\in B}\|h(x)\vv v_1\we\dots\we h(x)\vv v_r\|_\infty\,\ge\,
c^{\frac{r}{n+1}}\,\rho_{B}\,\min\{\psi^{-r},Q\} \, \ge  \,  1
$$
provided that $Q\ge Q_B$ for some sufficiently large $Q_B$ and $\psi\le \psi_B$ for some sufficiently small $\psi_B$.
If $r=n+1$, then trivially  the map
$$
x\mapsto\|h(x)\vv v_1\we\dots\we h(x)\vv v_r\|_\infty=\|\vv v_1\we\dots\we \vv v_r\|_\infty\ge 1
$$
is constant and hence $(C,\alpha)$-good for the same choice of $\alpha$ and some absolute constant $C>0$.

The upshot of the above is that all the conditions of Theorem~KM are met for any ball $B\subset\cU$, some constants $C,\alpha  > 0 $ and $\rho=1/(n+1)$. Therefore, by \eqref{vb3} and \eqref{vb2}, we obtain that
\begin{equation} \label{vb2+}
|B\setminus \cG(c,Q,\psi)| \; \le  \;   (n+1) \, C 6^{n+1} \left(\frac{c^{1/(n+1)}}{1/(n+1)}
\right)^\alpha |B|\,,
\end{equation}

\noindent where $\alpha =\frac1{2l-1}$. The latter inequality implies ~\eqref{theo_general_result_two+} for a suitably chosen $c>0$ that is independent of $B$. This thereby  completes the proof of the theorem.

\vspace*{2ex}

{\footnotesize

\noindent Victor Beresnevich: Department of Mathematics, University of York,\\
\phantom{Victor Beresnevich: }Heslington, York, YO10 5DD, UK\\
\phantom{Victor Beresnevich: }e-mail: vb8@york.ac.uk

\vspace{0mm}

\noindent Robert C. Vaughan: Department of Mathematics, Pennsylvania State University\\
\phantom{Robert C. Vaughan:  }University Park, PA 16802-6401, USA\\
\phantom{Robert C. Vaughan: }e-mail: rvaughan@math.psu.edu

\vspace{0mm}

\noindent Sanju Velani: Department of Mathematics, University of York,\\
\phantom{Sanju Velani: }Heslington, York, YO10 5DD, UK\\
\phantom{Sanju Velani: }e-mail: slv3@york.ac.uk

\vspace{0mm}

\noindent Evgeniy Zorin: Department of Mathematics, University of York,\\
\phantom{Evgeniy Zorin: }Heslington, York, YO10 5DD, UK\\
\phantom{Evgeniy Zorin: }e-mail: evgeniy.zorin@york.ac.uk

}

\end{document}